\documentclass[12pt]{amsart}
\usepackage{amsmath,amssymb,enumerate}
\newcommand{\IB}{\mathbb{B}}
\newcommand{\IC}{\mathbb{C}}
\newcommand{\IF}{\mathbb{F}}
\newcommand{\IM}{\mathbb{M}}
\newcommand{\IN}{\mathbb{N}}
\newcommand{\IZ}{\mathbb{Z}}

\newcommand{\cC}{\mathcal{C}}
\newcommand{\cH}{\mathcal{H}}
\newcommand{\cI}{\mathcal{I}}

\newcommand{\cO}{\mathcal{O}}
\newcommand{\cR}{\mathcal{R}}
\newcommand{\cU}{\mathcal{U}}
\newcommand{\cZ}{\mathcal{Z}}
\newcommand{\fS}{\mathfrak{S}}
\newcommand{\rC}{\mathrm{C}}
\newcommand{\rRC}{\mathrm{RC}}
\newcommand{\ve}{\varepsilon}
\newcommand{\vp}{\varphi}

\DeclareMathOperator{\dist}{dist}
\DeclareMathOperator{\GL}{GL}

\DeclareMathOperator{\sm}{s}
\newcommand{\CI}{{\mathcal C \mathcal I}}
\newcommand{\ip}[1]{\mathopen{\langle}#1\mathclose{\rangle}}
\newtheorem{thm}[subsection]{Theorem}
\newtheorem{prop}[subsection]{Proposition}
\newtheorem{cor}[subsection]{Corollary}
\newtheorem{lem}[subsection]{Lemma}
\theoremstyle{definition}
\newtheorem{defn}[subsection]{Definition}
\theoremstyle{remark}
\newtheorem{rem}[subsection]{Remark}
\title[Amenability for unitary groups]{Amenability for unitary groups of 
simple monotracial $\mbox{C}^*$-alge\-bras}
\author{Narutaka Ozawa}
\address{RIMS, Kyoto University, \mbox{606-8502} Japan}
\email{narutaka@kurims.kyoto-u.ac.jp}
\dedicatory{Dedicated to the memory of Eberhard Kirchberg}
\thanks{The author was partially supported by JSPS KAKENHI Grant Numbers 20H01806 and 20H00114}

\date{\today}

\begin{document}
\begin{abstract}
We prove the following two results. 
First, the isometry semigroup of a unital properly infinite 
nuclear $\mbox{C}^*$-alge\-bra is right amenable. 
Second, the unitary group of a unital simple monotracial 
$\mbox{C}^*$-alge\-bra whose tracial GNS representation 
is hyperfinite is skew-amenable in the weak topology. 
This answers in part a conjecture of Alekseev, Schmidt, 
and Thom and a question of Pestov. 
\end{abstract}
\maketitle
\section{Introduction}
We recall the cornerstone of the $\mathrm{C}^*$-alge\-bra 
theory that a $\mathrm{C}^*$-alge\-bra $A$ 
is nuclear (or amenable) if and only if the enveloping von Neumann 
algebra $A^{**}$ is hyperfinite (\cite{connes,ce,elliott}), 
which is equivalent to amenability of $\cU(A^{**})$ 
in the ultraweak topology (\cite{delaharpe}). 
In turn, amenability property of the unitary group $\cU(A)$ 
of a unital $\mathrm{C}^*$-alge\-bra $A$ has been drawing 
considerable attention of researchers. 
Recall that a topological group $G$ is said to be \emph{amenable} 
(resp.\ \emph{skew-amenable}) if 
there is a left-invariant (resp.\ right-invariant) mean 
on the space of right uniformly continuous bounded functions on $G$. 
It is known that $A$ is nuclear if and only if $\cU(A)$ is amenable 
in the weak topology (\cite{paterson}), essentially because $\cU(A)$ 
is dense in $\cU(A^{**})$ in the ultraweak topology. 
On the other hand, it is not clear when $\cU(A)$ 
with the norm topology is amenable. 
Note that norm amenability of $\cU(A)$ implies, 
in addition to nuclearity (\cite{connes:jfa}), 
that $A$ has the QTS property (i.e., every nonzero quotient of $A$ 
admits a tracial state). 
The converse is believed to hold, probably under some regularity assumptions. 
See \cite{ast} around this problem and the progress toward it. 
The purpose of this note is to prove the following two results. 

\begin{thm}\label{thm:1}
Let $A$ be a  a unital properly infinite $\mathrm{C}^*$-alge\-bra. 
Then the isometry semigroup $\cI(A)$ of $A$ is right amenable in the norm topology 
if and only if $A$ is nuclear.  
\end{thm}

Here, right amenability of $\cI(A)$ means existence of a right invariant 
mean on the space of uniformly continuous bounded functions on $\cI(A)$. 
We note that the assumption on proper infiniteness cannot be removed. 
Indeed, if $A$ is finite, then amenability of $\cI(A)=\cU(A)$ implies existence 
of a tracial state, which is not always the case (\cite{rordam:acta}). 
For the same reason, the conclusion of right amenability cannot be 
replaced with left amenability because it also implies existence of a tracial state 
(see Section~\ref{sec:2}). 

We turn our attention to the finite case. 
For a unital $\mbox{C}^*$-alge\-bra $A$, 
we denote by $T(A)$ the compact convex space of tracial states on $A$.
The $\mbox{C}^*$-alge\-bra $A$ is said to be \emph{monotracial} if $|T(A)|=1$. 
We write $\|a\|_\tau:=\tau(a^*a)^{1/2}$ for $a\in A$ and $\tau\in T(A)$ and 
define the \emph{uniform $2$-norm} on $A$ by
\[
\| a \|_{T(A)} := \sup\{ \|a\|_\tau : \tau\in T(A)\}.
\]

\begin{thm}\label{thm:2}
Let $A$ be a unital $\mathrm{C}^*$-alge\-bra with the QTS property and 
denote by $\cU(A)$ the unitary group of $A$. 
Consider the following conditions. 
\begin{enumerate}[{\rm(i)}]
\item\label{item:1} $\cU(A)$ is amenable in the uniform $2$-norm topology.
\item\label{item:2} $\cU(A)$ is skew-amenable in the weak topology. 
\item\label{item:3} For every $\tau\in T(A)$, the von Neumann 
algebra $\pi_\tau(A)''$ generated by the GNS representation 
$\pi_\tau$ for $\tau$ is hyperfinite.
\end{enumerate}
Then $(\mbox{\rm \ref{item:1}})$ $\Rightarrow$ $(\mbox{\rm \ref{item:2}})$ 
$\Rightarrow$ $(\mbox{\rm \ref{item:3}})$ holds.
If $A$ has only finitely many extremal tracial states, 
then $(\mbox{\rm \ref{item:3}})$ $\Rightarrow$ $(\mbox{\rm \ref{item:1}})$ holds. 
\end{thm}

This partly confirms/refutes a conjecture raised in \cite{ast}, 
where it is proved that $(\mbox{\rm \ref{item:2}})$ $\Rightarrow$ 
$(\mbox{\rm \ref{item:3}})$. 
We note that weak skew-amenability of $\cU(A)$ 
implies the QTS property of $A$ (\cite{ps,ast}).
The following corollary answers in the negative a question 
in \cite{pestov,js,ps} asking if skew-amenability 
implies amenability. 

\begin{cor}
Let $A$ be a unital simple monotracial $\mbox{C}^*$-alge\-bra and 
$\pi_\tau(A)''$ be the $\mathrm{II}_1$-factor  generated by the GNS 
representation $\pi_\tau$ for the unique tracial state $\tau$ on $A$. 
Then $\cU(A)$ is skew-amenable in the weak topology 
if and only if $\pi_\tau(A)''$ is hyperfinite. 

In particular, the unitary group $\cU(\cR)$ of the hyperfinite 
$\mathrm{II}_1$ factor $\cR$ (of any cardinality) 
is skew-amenable but not amenable in the weak topology. 
\end{cor}
\subsection*{Acknowledgment.}
I am grateful to Yasuhiko Sato, Benoit Collins, Yuhei Suzuki, and Mikael R{\o}rdam  
for many illuminating correspondences. 
This research was partially supported by 
JSPS KAKENHI Grant Numbers 20H01806 and 20H00114. 

This note is dedicated to the memory of Eberhard Kirchberg. 
I owe him a large intellectual debt. 
I spent a good part of my graduate study banging my head 
against his habilitationsschrift (Universit\"at Heidelberg, 1991),
which was very frustrating but equally rewarding. 
\section{Proof of Theorem~\ref{thm:1}}\label{sec:2}
\begin{defn}\label{defn:coliso}
Let $A$ be a unital $\mathrm{C}^*$-alge\-bra. 
A finite sequence $a=(a_1,\ldots,a_n)$ in $A$ is called a 
\emph{column isometry} if $a^*a:=\sum_i a_i^*a_i =1$, i.e., 
if $a$ is an isometry in $\IM_{n,1}(A)$. 
We identify a finite sequence $(a_1,\ldots,a_n)$ with 
$(a_1,\ldots,a_n,0,\ldots,0)$. 
The set of isometries (resp.\ column isometries) of $A$ 
is denoted by $\cI(A)$ (resp.\ $\CI(A)$). 
For $a\in\CI(A)$ and $s\in\cI(A)$, write $as:=(a_1s,\ldots,a_ns)\in\CI(A)$. 

For a finite sequence $a=(a_1,\ldots,a_n)$ in $A$, we put
\[
\| a \|_{\rC} := \|\sum_i a_i^*a_i\|^{1/2}
\mbox{ and }\| a \|_{\rRC}:=\|\sum_i a_ia_i^*\|^{1/2}+\|\sum_i a_i^*a_i\|^{1/2}.
\]
We note that the distance $\| a-b\|_{\rC}$ makes sense for 
any finite sequences $a$ and $b$, by padding them out 
with $0$s as necessary. 
\end{defn}

The following theorem is proved in \cite{ks} 
and Section 5 in \cite{sato}, where it is proved 
for $E\subset\cU(A)$ but the proof works verbatim 
for $E\subset\cI(A)$. 

\begin{thm}
A unital $\mathrm{C}^*$-alge\-bra $A$ is nuclear if and only 
if for every finite subset 
$E\subset\cI(A)$ and every $\ve>0$ there are a non-empty 
finite subset $F\subset\CI(A)$ and 
permutations $\{\rho_s : s\in E\}$ on $F$ such that 
$\|a s - \rho_s(a) \|_{\rC} < \ve$ for every $s\in E$ and $a\in F$. 
\end{thm}

If $A$ is moreover properly infinite, 
then one can take isometries $s_1,s_2,\ldots$ with mutually 
orthogonal ranges and replace $\CI(A)$ with $\cI(A)$ 
via the right $\cI(A)$-equivariant isometric map 
\[
\CI(A)\ni(a_1,\ldots,a_n)\mapsto\sum_i s_ia_i \in\cI(A).
\]
By taking a limit point of the uniform probability measures 
on suitable $F$s, 
one obtains a right invariant mean on $\cI(A)$. 
This proves the `if' part of Theorem~\ref{thm:1}. 
The proof of the `only if' part is standard (\cite{connes:jfa}). 
Take a universal representation $A\subset\IB(\cH)$ and consider 
for each $x\in\IB(\cH)$ and $\xi,\eta\in\cH$ 
the uniformly continuous bounded function 
\[
f_{x,\xi,\eta}\colon \cI(A)\ni s \mapsto \ip{s^*xs\xi,\eta} \in \IC.
\] 
If $\cI(A)$ admits a right invariant mean $m$, 
then by the Riesz representation theorem 
there is $\Phi(x) \in \IB(\cH)$ that satisfies 
$m(f_{x,\xi,\eta})=\ip{\Phi(x)\xi,\eta}$ for every $\xi,\eta\in\cH$. 
It is not hard to see that $x\mapsto\Phi(x)$ is a conditional expectation 
from $\IB(\cH)$ onto $A'$. 
This implies that $A^{**}\cong (A')^{\mathrm{op}}$ 
is injective (hyperfinite) and hence that $A$ is nuclear (\cite{connes,ce}). 
In passing, we observe that if $\cI(A)$ admits a left invariant mean $m'$, 
then for any unit vector $\xi$, 
the map $x\mapsto m'(f_{x,\xi,\xi})$ defines a tracial state on $A$ 
(in fact an $A$-central state on $\IB(\cH)$).

On the other hand, if $A$ is moreover with the QTS property 
(as opposed to proper infiniteness), then by 
Dixmier's averaging (Theorem~1 in \cite{dix}),  
there is a finite sequence $w_1,\ldots,w_k\in \cU(A)$ that satisfies
\[
\|k^{-1}\sum_{i,j} w_ja_ia_i^*w_j^* - 1 \|<\ve
\mbox{ and }
\|k^{-1}\sum_{i,j} w_j(a_is - \rho_s(a)_i)(a_is - \rho_s(a)_i)^*w_j^*\|<\ve
\] 
for every $s\in E$ and $a\in F$. 
Thus by replacing each $a\in F$ with $(k^{-1/2}w_ja_i)_{i,j}$ 
and retaining $\{\rho_s\}$, we obtain the following.
\begin{cor}\label{cor}
Let $A$ be a unital nuclear $\mathrm{C}^*$-alge\-bra with the QTS property, 
Then, for every finite subset 
$E\subset\cU(A)$ and every $\ve>0$ 
there are a non-empty finite subset $F\subset\CI(A)$ and 
permutations $\{\rho_u : u\in E\}$ on $F$ such that 
$\|\sum_i a_ia_i^* - 1 \|<\ve$ and 
$\|a u - \rho_u(a) \|_{\rRC} < \ve$ for every $u\in E$ and $a\in F$. 
\end{cor}
\section{Proof of Theorem~\ref{thm:2}}
\begin{proof}[Proof of Theorem~\ref{thm:2}]
We consider the following strengthening of condition $(\mbox{\rm \ref{item:2}})$:
\begin{itemize}
\item[$(\mbox{\rm \ref{item:2}}')$] 
$\cU(A)$ is skew-amenable in the skew-strong topology. 
\end{itemize}
Here the \emph{skew-strong topology} on $A$ is given by 
the directed family of semi-norms $\|\,\cdot\,\|_\vp$, $\vp\in A^*_+$, 
where $\| a \|_\vp := \vp(aa^*)^{1/2}$ for $a\in A$. 
Be aware that it is \emph{not} the more common $\| a \|_\vp=\vp(a^*a)^{1/2}$.
As $A^*$ is spanned by $A^*_+$, the Cauchy--Schwarz inequality 
implies that the skew-strong topology is finer than the weak topology. 
Thus $(\mbox{\rm \ref{item:2}}')$ $\Rightarrow$ $(\mbox{\rm \ref{item:2}})$ holds. 

Let's assume $(\mbox{\rm \ref{item:1}})$ and prove $(\mbox{\rm \ref{item:2}}')$. 
By Theorem 4.5 in \cite{st}, condition $(\mbox{\rm \ref{item:1}})$ means 
that for every finite subset $E\subset\cU(A)$ and $\ve>0$ 
there are a non-empty finite subset $F\subset\cU(A)$ and permutations 
$\{\rho_u : u\in E\}$ on $F$ that satisfy 
\[
|F|^{-1}\sum_{v\in F} \sum_{u\in E} \| vu - \rho_u(v) \|_{T(A)}^2 <\ve.
\]
Hence by Dixmier's averaging (Theorem~1 in \cite{dix}), there is a finite 
sequence $w_1,\ldots,w_k$ in $\cU(A)$ such that 
\[
\| k^{-1}\sum_{j=1}^k w_j \Bigl( |F|^{-1}\sum_{u\in E,\, v\in F}
(v u - \rho_u(v))(v u - \rho_u(v))^* \Bigr) w_j^*\| < \ve.
\]
It follows that for every state $\vp$ on $A$ there is $j$ such that 
\[
|F|^{-1}\sum_{u\in E,\, v\in F}\|w_jv u - w_j\rho_u(v) \|_\vp^2 < \ve.
\]
By replacing $F$ with $\{ w_jv : v\in F\}$ and retaining $\{\rho_u\}$, 
one sees skew-amenability of $\cU(A)$ in $\|\,\cdot\|_\vp$. 
Since the semi-norms $\| \,\cdot\|_\vp$, $\vp\in A^*_+$, are directed, 
this proves $(\mbox{\rm \ref{item:2}}')$.

The implication 
$(\mbox{\rm \ref{item:2}})$ $\Rightarrow$ $(\mbox{\rm \ref{item:3}})$ is 
Proposition 4.4 in \cite{ast}. 
We prove $(\mbox{\rm \ref{item:3}})$ $\Rightarrow$ $(\mbox{\rm \ref{item:1}})$ 
assuming that $A$ has only finitely many extremal tracial states $\tau_1,\ldots,\tau_k$.  
We put $\tau:=k^{-1}\sum_{j=1}^k\tau_j \in T(A)$ and observe that  
$\| a \|_{T(A)}^2 \le k\| a\|_\tau^2$ for every $a\in A$. 
By Kaplansky's density theorem, 
$\cU(A)$ is $\|\,\cdot\,\|_\tau$-dense in $\cU(\pi_\tau(A)'')$. 
The rest is standard: amenability of $\cU(\pi_\tau(A)'')$ 
is inherited by the dense subgroup $\cU(A)$. 
We include the proof for completeness. 
Let a finite subset $E\subset\cU(A)$ and $\ve>0$ be given. 
Since $\pi_\tau(A)''$ is hyperfinite, there is a finite-dimensional 
von Neumann subalgebra $B\subset \pi_\tau(A)''$ such that 
$\dist_\tau(u,\cU(B))<\ve$ for every $u\in E$, i.e.,  
for each $u\in E$, there is $\tilde{u}\in\cU(B)$ with $\| \tilde{u} - u \|_\tau<\ve$.
Since $\cU(B)$ is a compact group, there is a finite subset 
$\{ \tilde{v}_1,\ldots,\tilde{v}_n\} \subset \cU(B)$ 
and permutations $\{\rho_u : u\in E\}$ on $\{1,\ldots,n\}$ such that 
$\| \tilde{v}_i \tilde{u} - \tilde{v}_{\rho_u(i)} \|_\tau <\ve$ 
for every $\tilde{u} \in \tilde{E}$ and $i\in\{1,\ldots,n\}$. 
For each $i$, take $v_i\in\cU(A)$ with $\| v_i - \tilde{v}_i \|_\tau<\ve$ 
and put $F:=\{ v_1,\ldots,v_n\}$. 
Since 
\[
\| v_iu - \tilde{v}_i\tilde{u} \|_\tau
 \le \| (v_i - \tilde{v}_i)u\|_\tau + \| \tilde{v}_i(u-\tilde{u})\|_\tau \le 2\ve
\]
(this is where the tracial property is indispensable), 
one has $\| v_iu - v_{\rho_u(i)} \|_\tau \le 4\ve$ for every $u\in E$ and $i$. 
By adjusting $\ve>0$, we are done.
\end{proof}

\begin{rem}
Since every norm separable subset of $\cR$ is contained in 
a separable simple monotracial $\mathrm{C}^*$-sub\-alge\-bra 
of $\cR$ (see e.g., Lemma~9 in \cite{dix}), 
there exists a unital separable simple monotracial 
non-exact (\cite{kirchberg}) $\mbox{C}^*$-alge\-bra $A$ 
whose unitary group $\cU(A)$ is skew-amenable but not 
amenable in the weak topology. 
\end{rem}
\section{Further results that may be useful in the future}
The rest of this note handles the case of $\mbox{C}^*$-alge\-bras 
with infinitely many extremal tracial states. 
Exactness plays a crucial role, 
as it assures certain commutativity of ultraproduct 
and tensor product (\cite{kirchberg}). 
We first collect some useful facts about the free semi-circular systems.  

We recall the free semi-circular system $\{ s_i : i=1,2,\ldots\}$.
Let $\cO_\infty$ be the Cuntz algebra generated by 
isometries $l_i$ with mutually orthogonal ranges 
and let $s_i := l_i+l_i^*$.
Then, $\cC := \mbox{C}^*(\{s_i: i=1,2,\ldots\})$ is $*$-iso\-mor\-phic 
to the reduced free product of the copies of $C([-2,2])$ 
with respect to the Lebesgue measure 
(see Section 2.6 in \cite{vdn}), and the corresponding tracial
state $\tau_\cC$ coincides with the restriction to $\cC$ of 
the vacuum state on $\cO_\infty$.
We note that the reduced free group $\mbox{C}^*$-alge\-bra 
$\mbox{C}^*_{\mathrm{r}}\IF_d$ and 
the free semi-circular $\mbox{C}^*$-alge\-bra $\cC$ 
embed into each other, because $\mbox{C}^*_{\mathrm{r}}\IZ$ 
and $C([-2,2])$ embed into each other. 
Also, the countable free groups of rank at least two embeds 
into each other as groups. 

\begin{thm}\label{thm:fsemb}
Let $\cZ$ denote the Jiang--Su algebra and $\cZ^\omega$ 
the ultrapower w.r.t.\ a free ultrafilter on $\IN$. 
Then, the free semi-circular $\mbox{C}^*$-alge\-bra $\cC$ 
embeds into $\cZ^\omega$. 
\end{thm}
\begin{proof}
Recall that the Jiang--Su algebra $\cZ$ is a (unique) 
simple monotracial $\mbox{C}^*$-alge\-bra which arises 
as an inductive limit of some prime dimension drop 
$\mbox{C}^*$-alge\-bras 
\[
\{ f\in C([0,1],\IM_{p(n)}\otimes\IM_{q(n)}) : 
f(0)\in\IM_{p(n)}\otimes\IC1,\ 
f(1)\in\IC1\otimes\IM_{q(n)}\},
\]
where $(p(n),q(n))$ are pairs of relatively prime numbers. 
We may assume $p(n)\gg q(n)$. 
We take i.i.d.\ G.U.E.\ random matrices 
$x_1(n),x_2(n) \in \IM_{p(n)}$ and $y_1(n),y_2(n) \in \IM_{q(n)}$. 
Then by \cite{ht,cgp,pisier} 
(or more advanced \cite{bc1,bc2} if we do not want to assume $p(n)\gg q(n)$),
the tuple $(x_1(n)\otimes 1_{q(n)}, x_2(n)\otimes 1_{q(n)}, 1_{p(n)}\otimes y_1(n), 1_{p(n)}\otimes y_2(n))$ strongly converges to 
$(s_1\otimes 1,s_2\otimes1,1\otimes s_1,1\otimes s_2)$ 
in $\cC_2\otimes\cC_2$, where $\cC_2$ is the 
$\mbox{C}^*$-alge\-bra generated by 
the free semicircular system $\{ s_1, s_2\}$;
in other words, there is an embedding 
\[
\cC_2\otimes \cC_2\hookrightarrow
 \Bigl(\prod_n \IM_{p(n)}\otimes\IM_{q(n)}\Bigr)/\omega
\] 
of $\cC_2\otimes\cC_2$ into 
the norm ultraproduct $(\prod_n \IM_{p(n)}\otimes\IM_{q(n)})/\omega$.
Since $C([0,1])$ is exact, we may view by \cite{kirchberg}  that 
\[
C([0,1],\Bigl(\prod_n \IM_{p(n)}\otimes\IM_{q(n)}\Bigr)/\omega)
  \subset \Bigl(\prod_n C([0,1],\IM_{p(n)}\otimes\IM_{q(n)})\Bigr)/\omega
\]
and thus
\[
B:=\{f \in C([0,1],\cC_2\otimes\cC_2) : f(0)\in \cC_2\otimes\IC1,\ f(1)\in\IC1\otimes \cC_2\}
\hookrightarrow \cZ^\omega.
\]
For each $i=1,2$ we take $h_i=h_i^* \in \mathrm{C}^*(\{1,s_i\})$ 
such that the unitary element $u_i:=\exp(\sqrt{-1} h_i)$ 
satisfies $\tau_{\cC_2}(u_i^n)=0$ for all $n\neq0$. 
We put $u_i(t):= \exp( t \sqrt{-1} h_i)$ which connects  
$u_i(0)=1$ to $u_i(1)=u_i$. 
We define $g_i\in C([0,1],\cC_2\otimes\cC_2)$ by 
$g_i(t) := u_i(1) \otimes u_i(2t)$ for $t\in[0,1/2]$ 
and $g_i(t)=u_i(2- 2t)\otimes u_i(1)$ for $t\in[1/2,1]$. 
Then $g\in B$ and, for each $t$, the pair $\{g_1(t), g_2(t)\}$ 
is unitarily equivalent to the standard generating pair of 
$\mathrm{C}^*_{\mathrm{r}}\IF_2 \otimes \IC1$, by Fell's absorption principle.  
Thus, $\{g_1, g_2\}$ itself generates a copy of 
$\mathrm{C}^*_{\mathrm{r}}\IF_2$ inside $\cZ^\omega$.
\end{proof}

For every finite sequence $a=(a_1,\ldots,a_n)$ 
in a $\mbox{C}^*$-alge\-bra $A$, we write
\[
\sm(a) := \sum_i s_i\otimes a_i \in \cC\otimes A, 
\]
where $\{ s_i\}$ is a free semi-circular system. 
We note that 
\[
\sm(a) = \sum_i l_i \otimes a_i + (\sum_i l_i \otimes a_i^*)^* 
 =: S+T^*
\]
in $\cO_\infty\otimes A$ and that $S$ and $T$ satisfy
\[
S^*S = 1 \otimes \sum_i a_i^*a_i\mbox{ and } T^*T = 1 \otimes \sum_i a_ia_i^*.
\]
Thus, in particular (see Definition~\ref{defn:coliso} for $\|\,\cdot\,\|_{\rRC}$)
\[
\| \sm(a) \| \le \|a\|_{\rRC}.
\]
\begin{lem}\label{lem:st}
For every $\ve>0$, there is $\delta>0$ that satisfies the following. 
For every finite sequence $a$ in a $\mbox{C}^*$-alge\-bra $A$ with $T(A)\neq\emptyset$, 
if $\|\sum_i a_i^*a_i -1\|<\delta$ and $\|\sum_i a_ia_i^*-1\|<\delta$, then 
\[
\| 1_{[0,\delta)}(|\sm(a)|) \|_{T(\cC\otimes A)}<\ve.
\]
Here $1_{[0,\delta)}(|\sm(a)|) \in (\cC\otimes A)^{**}$ is the spectral projection 
for $|\sm(a)|$ corresponding to $[0,\delta)$. 
\end{lem}
\begin{proof}
Note that $T(\cC\otimes A)=\{\tau_{\cC}\otimes\tau_A : \tau_A\in T(A)\}$. 
Suppose that the conclusion were false. 
Then, there is $\ve>0$ such that for every $n$ there are a finite sequence 
$a_n$ in $A_n$ and $\tau_n\in T(A_n)$ that satisfy
$\|\sum_i a_{n,i}^*a_{n,i} -1\|<1/n$, $\|\sum_i a_{n,i}a_{n,i}^*-1\|<1/n$, 
and $(\tau_{\cC}\otimes\tau_{A_n})(1_{[0,1/n)}(|\sm(a_n)|))\geq\ve$. 
We define a continuous function $f$ 
by $f=1$ on $[0,1/n]$, $f=0$ on $[1/(n-1),\infty)$, 
and linear on $[1/n,1/(n-1)]$.
Let $\psi$ denote the vacuum state on $\cO_\infty$ that extends $\tau_{\cC}$ 
and put $\vp_n:=\psi\otimes\tau_{A_n}$ on $\cO_\infty\otimes A_n$. 
Then, $S_n := \sum_i l_i \otimes a_{n,i}$ 
and $T_n:=\sum_i l_i \otimes a_{n,i}^*$ satisfy
$\|S_n^*S_n-1\|<1/n$, $\|T_n^*T_n-1\|<1/n$, $\vp_n(T_nT_n^*)=0$, 
and $\vp_n( f_m(|S_n+T_n^*|) ) \geq \ve$
for all $n$ and $m$ with $n\geq m$. 
Thus by passing to an ultralimit, one obtains isometries $S$, $T$ and a state $\vp$ 
such that 
$\vp(TT^*)=0$ and $\vp(f_m(|S+T^*|))\geq\ve$ for all $m$. 
By the GNS construction, we may assume that $\vp$ is the vector 
state associated with a unit vector $\xi$. 
Then, $\lim_m \vp(f_m(|S+T^*|))=\| P\xi\|^2$, where $P$ 
is the orthogonal projection onto the kernel of $S+T^*$. 
Since $S$ and $T$ are isometries, $(S+T^*)\eta=0$ is equivalent to 
that $-TS\eta=\eta$. 
Hence $P$ is a WOT-limit point of $( k^{-1}\sum_{j=1}^k (-TS)^j )_k$. 
However, since $\|T^*\xi\|^2=\vp(TT^*)=0$, this implies $P\xi=0$, 
contradicting that $\|P\xi\|^2\geq\ve$. 
\end{proof}

In particular, if $\sum_i a_i^*a_i =1 = \sum_i a_ia_i^*$, then 
$\| 1_{\{0\}}(|\sm(a)|)\|_{T(\cC\otimes A)}=0$.
However, $\sm(a)=S+T^*$ is not invertible since 
$S$ and $T$ are proper isometries (in which case $-1$ 
is an approximate eigenvalue of $TS$). 
The author does not know whether 
$\| 1_{\{0\}}(|\sm(a)|)\|_{T(\cC\otimes A)}=0$ 
holds as soon as 
$\|\sum_i a_i^*a_i -1\|<1/2$ and $\|\sum_i a_ia_i^*-1\|<1/2$. 
\begin{thm}\label{thm:amenunif}
Let $A$ be a unital simple finite nuclear $\cZ$-stable $\mathrm{C}^*$-alge\-bra. 
Then for every finite subset 
$E\subset\cU(A)$ and $\ve>0$, 
there are a non-empty finite subset $F\subset\cU(A)$ 
and permutations $\{\rho_u : u\in E\}$ on $F$ such that 
$\|v u - \rho_u(v) \|_{T(A)} < \ve$ 
for every $u\in E$ and $v\in F$. 
In particular, $\cU(A)$ is amenable in the uniform $2$-norm topology. 
\end{thm}

\begin{proof}
Let a finite subset $E\subset\cU(A)$ and $\ve>0$ be given. 
We may assume that $A=\cZ\otimes A_0$, $A_0\cong A$, 
and $E\subset \IC1\otimes A_0$. 
We embed $\cZ\otimes A_0$ into the norm ultrapower 
$(\cZ\otimes A_0)^\omega$.
We take an embedding of the free semi-circular $\mathrm{C}^*$-alge\-bra
$\cC$ into $\cZ^\omega$ and view (by exactness of $A_0$) 
\[
\cC \otimes A_0\subset \cZ^\omega \otimes A_0 \subset (\cZ\otimes A_0)^\omega.
\]

Take $\delta>0$ as in Lemma~\ref{lem:st} and 
take $F\subset\CI(A_0)$ and $\{\rho_u : u\in E\}$ 
as in Corollary~\ref{cor} for $E$ and $\kappa$ 
(instead of $\ve$ there), where $\kappa>0$ is 
a sufficiently small number which is specified later.
Hence $\|1_{[0,\delta)}(|s(a)^*|)\|_{T(\cC\otimes A_0)}<\ve$ 
for every $a\in F$. 
Further,  
$\sm(a)\in\cC\otimes A_0$ 
satisfies, for every $u\in E$ and $a\in F$, that 
\[
\|\sm(a)\|\le \|a\|_{\rRC}\le 3
\mbox{ and }
\| \sm(a)u - s(\rho_u(a)) \|\le\| au-\rho_u(a)\|_{\rRC}<\kappa.
\]
Consider the polar decomposition $\sm(a)=|\sm(a)^*|w(a)$ 
of $\sm(a)^*$. 
Note that $w(a) \in (\cC\otimes A_0)^{**}$ 
may not belong to $\cC\otimes A_0$, 
but they satisfy $w(au)=w(a)u$. 
Since $A_0$ is simple, finite, and $\cZ$-stable, $\cC\otimes A_0$ has 
stable rank one by Theorem~6.7 in \cite{rordam:ijm}. 
Hence, $\sm(a)\in \overline{\GL(\cC\otimes A_0)}$ 
and by \cite{pedersen, rordam:annals} 
there is a unitary element $v(a)$ in $\cC\otimes A_0$ such that 
$1_{[\delta,\infty)}(|\sm(a)^*|) v(a)= 1_{[\delta,\infty)}(|\sm(a)^*|) w(a)$. 
Let $g(t)=\min\{\delta^{-1} t,1\}$ be the continuous function on $[0,\infty)$. 
We may assume that $\kappa>0$ is small enough so that 
$\| g(|x^*|)w(x) - g(|y^*|)w(y) \| < \ve$ whenever $x=|x^*|w(x)$ and $y=|y^*|w(y)$ 
(polar decompositions) satisfy $\|x\|,\,\|y\|\le 3$ and $\| x- y \|<\kappa$. 
We write $x\approx_{\gamma} y$ if $\|x-y\|_{T(\cC\otimes A_0)}<\gamma$. 
Since $\sm(au)=\sm(a)u=|\sm(a)^*|w(a)u$, 
$\sm(\rho_u(a))=|\sm(\rho_u(a))^*|w(\sm(\rho_u(a)))$, and 
$\|\sm(au) - \sm(\rho_u(a))\|<\kappa$, 
one has
\begin{align*}
v(a)u \approx_{2\ve} g(|\sm(a)^*|)w(a)u
 \approx_{\ve} g(|\sm(\rho_u(a))^*|)w(\rho_u(a))
 \approx_{2\ve}  v(\rho_u(a)),
\end{align*}
or equivalently
\[
\| v(a)u - v(\rho_u(a)) \|_{T(\cC\otimes A_0)} < 5\ve 
\]
for every $u\in E$ and $a\in F$.

Recall that $\cC\otimes A_0 \subset (\cZ\otimes A_0)^\omega$ 
and write $v(a) = (v_n(a))_{n\to\omega} \in (\cZ\otimes A_0)^\omega$, 
where $v_n(a)\in\cU(\cZ\otimes A_0)=\cU(A)$ for every $a\in F$ and $n\in\IN$.  
For every $u\in E$ and $a\in F$, one has 
\[
\{ n\in\IN : \|v_n(a)u - v_n(\rho_u(a)) \|_{T(A)} < 5\ve \} \in \omega,
\]
for otherwise 
$\{ n\in\IN : \|v_n(a)u - v_n(\rho_u(a)) \|_{\tau_n} \geq 5\ve\}\in\omega$ 
for some $(\tau_n)_n\in T(A)^{\IN}$, which implies 
$\| v(a)u - v(\rho_u(a)) \|_{\tau_\omega}\geq 5\ve$ for  
the ultraproduct tracial state $\tau_\omega$ of $(\tau_n)_n$.
It follows that there is $n\in\IN$ that satisfies 
$\|v_n(a)u - v_n(\rho_u(a)) \|_{T(A)} < 5\ve$ 
simultaneously for all $u\in E$ and $a\in F$.
\end{proof}

The simplicity assumption in lieu of the QTS property 
is used in the above proof only to have stable rank one. 
This begs the question: is it true that $\sm(a) \in \overline{\GL(\cC\otimes A)}$ 
for every finite sequence $a$ in a $\mathrm{C}^*$-alge\-bra with the QTS property?  
Incidentally, it seems that one can push the above proof and replace 
the simplicity assumption with the QTS property, by manipulating 
on a large matrix algebra $\IM_n(\cC\otimes A_0)$ 
and by suitably embedding $C_0((0,1],\IM_n(\cC\otimes A_0))$ 
into $\cZ\otimes \cC\otimes A_0$, without much affecting 
the uniform $2$-norm.

\end{document}